\documentclass[12pt]{amsart}
\input xy  \xyoption{all}
\usepackage[latin2]{inputenc}
\usepackage{enumerate}
\usepackage{amssymb}

\usepackage[T2A]{fontenc}
\newcommand{\Zhe}{\mbox{\usefont{T2A}{\rmdefault}{m}{n}\CYRZH}}

\DeclareMathOperator\N{\mathbb N}

\renewcommand{\C}{\mathbb C}
\DeclareMathOperator\Z{\mathbb Z}

\renewcommand{\P}{\mathbf P}
\DeclareMathOperator\Hom{Hom}
\DeclareMathOperator\tp{\mathrm tp}
\DeclareMathOperator\Tp{\mathrm Tp}
\DeclareMathOperator\gr{\mathrm Gr}
\DeclareMathOperator\gl{\mathrm GL}
\DeclareMathOperator\GL{\mathrm GL}
\DeclareMathOperator\SL{\mathrm SL}

\DeclareMathOperator\codim{codim}

\DeclareMathOperator\jac{Jac}
\DeclareMathOperator\noc{\mathbf{Noc}}
\DeclareMathOperator\pl{\mathbf{Pl}}

\newcommand{\bigw}{\bigwedge\nolimits}
\def\iso{\cong}

\newtheorem{fact}{Fact}[section]

\newtheorem{theorem}[fact]{Theorem}
\newtheorem{definition}[fact]{Definition}
\newtheorem{example}[fact]{Example}
\newtheorem{rremark}[fact]{Remark}\newenvironment{remark}{\begin{rremark} \rm}{\end{rremark}}
\newtheorem{proposition}[fact]{Proposition}

\title[Nets of conics]{Equivariant and invariant theory of nets of conics with an application to Thom polynomials}

\author{M. Domokos}
\address{R\'enyi Institute of Mathematics, Hungarian Academy of Sciences, Budapest, Hungary}
\email{domokos@renyi.mta.hu}

\author{L. M. Feh\'er}
\address{Department of Analysis, Eotvos University Budapest, Hungary}
\email{lfeher@renyi.mta.hu}

\author{R. Rim\'anyi}
\address{Department of Mathematics, University of North Carolina at Chapel Hill, USA}
\email{rimanyi@email.unc.edu}

\thanks{The first named author is supported by OTKA grants NK 81203 and K 101515. The second author is supported by the OTKA grants 72537 and 81203. The third author is supported by NSA grant CON:H98230-10-1-0171.}

\begin{document}

\begin{abstract}
Two parameter families of plane conics are called nets of conics. There is a natural group action on the vector space of nets of conics, namely the product of the group reparametrizing the underlying plane, and the group reparametrizing the parameter space of the family. We calculate equivariant fundamental classes of orbit closures. Based on this calculation we develop the invariant theory of nets of conics. As an application we determine Thom polynomials of contact singularities of type (3,3). We also show how enumerative problems---in particular the intersection multiplicities of the determinant map from nets of conics to plane cubics---can be solved studying equivariant classes of orbit closures.
\end{abstract}

\maketitle

\section{Introduction}
The Thom polynomial is a multivariable polynomial associated with a singularity, governing its global behavior. Thom polynomials of certain singularities are known, mostly of those with small codimension, or with some other simple structure. The Thom polynomials of the simplest so-called $\Sigma^3$ singularities are not known. The first objective of this paper is to calculate the Thom polynomials of these simplest $\Sigma^3$ singularities. These singularities are closely related with nets (ie. two-parameter families) of (plane) conics. Recent work of the second and the third authors \cite{ts_loc} reduces the calculations of these Thom polynomials to the calculation of the equivariant cohomology classes of the corresponding classes of nets of conics.

While calculating the equivariant cohomology classes of nets of conics we realized that most of the classical results on the invariant theory of nets of conics can be reproduced using our equivariant techniques. We included the invariant theory of nets of conics and its relation with the equivariant theory for two reasons. First, this makes the paper more self-contained, and second, we find the relation between equivariant and invariant theory interesting, and possibly applicable in the future for more general representations.

Another byproduct of our calculations is a list of intersection multiplicities for the determinant map (assigning a plane cubic to a net of conics---its discriminant). Intersection multiplicities are difficult to calculate in general. We will define and compute equivariant intersection multiplicities, and show that they agree with the non-equivariant ones.

\section{Structure of the paper}

Two parameter families of plane conics are called nets of conics. There is a natural group action on the vector space of nets of conics, namely the product of the group of linear reparametrizations of the underlying plane, and the group of linear reparametrizations of the parameter space of the family.

In Section~\ref{sec:class} we develop the equivariant cohomology theory of nets of conics with respect to this action. The results of the section are summarized in Theorems \ref{thm:nets_equiv_classes} and \ref{thm:codim1tp}.

In Section \ref{sec:ts}, using results of \cite{ts_loc}, we determine Thom polynomials of contact singularities of type (3,3). The main results of the paper is Theorem \ref{main} and \ref{DE}.

In Section \ref{sec:inv-of-noc} we develop the invariant theory of nets of conics. While this theory is mostly known, we emphasize that the equivariant point of view gives a conceptional way to approach invariant theory. This section however is relatively independent of the rest of the paper, it does not use cohomology theory. It is known that the corresponding ring of invariants is generated by two algebraically independent invariants having degree  $6$ and $12$, see Section \ref{sec:inv-of-noc} for references. We give a formula  for the degree 6 invariant (already known to Salmon and Sylvester) in terms of the Plücker coordinates as a pull-back of a degree 2 invariant (\ref{eq:i2}) of the Plücker space. The expression (\ref{eq:i2}) for this  degree 2 invariant appears to be new.


In Section \ref{sec:hierarchy} we describe the hierarchy (see Figure \ref{fig:hierarchy}) of the orbits of nets of conics. This hierarchy was first obtained by different techniques in \cite{wall:nets}. We show how the cohomological data obtained in Section \ref{sec:class} through the notion of {\em incidence class} can recover the result on hierarchy. This method of reducing hierarchy to equivariant cohomology calculations may be applied in the future to other representations with small GIT quotients. The technique is particularly promising in the emerging field of Geometric Complexity Theory, see \cite{gct}.

In Section \ref{sec:mult} we show how enumerative problems---in particular the intersection multiplicities of the determinant map from nets of conics to plane cubics---can be solved studying the equivariant cohomology classes of the orbits. The main result of the section is Theorem \ref{multiplicities}. We also explain that in some cases the intersection multiplicities agree with the algebraic multiplicities. We hope that these calculations may lead to the determination of which orbit-closures of the plane cubics are Cohen-Macaulay (see Remark \ref{alg-mult}). For completeness we included the multiplicities of the induced map between the corresponding GIT quotients though these results are probably known.

\smallskip

Throughout the paper we work in the complex algebraic category, hence in particular, $\GL(U)$ means the  group of complex linear transformations of the complex vector space $U$, and $\GL_n$ denotes $\GL(\C^n)$. Cohomology will be considered with integer coefficients.

\smallskip

The authors are grateful to M. Kazarian for several useful discussions on Thom polynomials and to C. T. C. Wall for very valuable comments on nets of conics. Additionally the authors would like to thank I. Dolgachev, J. Chipalkatti, L. Oeding and P. Frenkel for useful conversations on the topics in this paper.

\section{Classification of orbits of nets and their equivariant classes}\label{sec:class}

\subsection{Orbits of nets of conics}

Let $S^2U$ denote the second symmetric power of the vector space (or representation) $U$. Consider the vector spaces $U=\C^3$ and $V=\C^3$. The main object of this paper is the $\GL(U)\times \GL(V)$ representation $\noc=\Hom(S^2U,V)$.

Through the natural isomorphism $\Hom(S^2U,V)=\Hom(V^*\negthinspace, S^2U^*)$, elements of this vector space are families of homogeneous degree 2 polynomials on $U$ parameterized by the vector space $V^*$. Lines in $S^2U^*$ determine conics in the projective plane $\P(U)$, hence elements of $\Hom(V^*\negthinspace, S^2U^*)$ are 2-parameter families ({\em nets}) of plane conics ($\noc$ stands for {\em nets of conics}). The $\GL(U)$ action reparametrizes the underlying plane $\P(U)$, and the $\GL(V)$ action reparametrizes the parameter space $V^*$\negthinspace.

There is a natural stratification $\Sigma^2\cup\Sigma^1\cup\Sigma^0$ of $\noc-\{0\}$, according to corank. Geometrically the strata correspond to conics, pencils of conics (ie. 1-parameter families of conics), and (proper) nets of conics, respectively.  The orbit structure of conics and pencils of conics is widely known. The classification of orbits of proper nets of conics is given in \cite{wall:nets} for the codimension$>1$ cases and in \cite{wall-duplessis} for the family of codimension 1 orbits. We will use their notations.

The list of codimension$>1$ orbits is given in the first 3 columns of Table \ref{symmetries} with the following conventions. Column 1 is the name of the orbit, column 2 is its codimension, and column 3 names three plane conics that span the image of $\phi\in \Hom(V^*, S^2U^*)$. Here we used the letters  $x,y,z$ for the coordinates on $U$.

\subsection{Equivariant classes} $G$-invariant subvarieties (eg. orbit closures) represent cohomology classes in the equivariant cohomology ring of a $G$-representation. We want to determine the equivariant classes $[\overline{\eta}]\in H^*_{\GL(U)\times \GL(V)}(\noc)\iso \Z[u_1, u_2, u_3,v_1,v_2,v_3]$ for the orbits $\eta\subset \noc$. Here $u_1, u_2, u_3$ and $v_1,v_2,v_3$ denote the Chern classes of the groups $\GL(U)$ and $\GL(V)$ respectively. These classes contain a lot of geometric information as we will show later.

For the codimension$>1$ orbits we use the method of {\em restriction equations} of \cite[Thm.2.4]{rrtp}, see also \cite[Sect.3]{cr}, so we need the symmetries of these orbits. More precisely, we need only a maximal torus of their stabilizer subgroups. These elementary calculations can be reduced to the level of Lie algebras. The results are summarized in the  ``symmetry'' column of Table \ref{symmetries}.

\begin{table}\caption{Codimension$>1$ orbits and symmetries}\label{symmetries}\begin{tabular}{|l|l|l|l|l|l|}

\hline
$$&cd & representative& symmetry&Poincar\'e&$\delta$\\
\hline
$\Sigma^0$&&&&&\\
\hline
$C$      & 2 & $y^2+2xz,2yz,-x^2$        & $(2\alpha,\alpha+\beta,2\beta),(2\alpha+2\beta, \alpha+3\beta, 4\alpha)$      & 1,1  &$\nu$\\
$D$      & 2 & $x^2, y^2, z^2+2xy$       & $(2\alpha,2\beta,\alpha+\beta),(4\alpha,4\beta,2\alpha+2\beta)$               & 1,2 &$\theta$\\
$D^*$    & 2 & $2xz, 2yz, z^2+2xy$       & $(2\alpha,2\beta,\alpha+\beta),(3\alpha+\beta,\alpha+3\beta, 2\alpha+2\beta)$ & 1,2 &$\theta$\\
$E$      & 3 & $x^2,y^2,z^2$             & $(\alpha,\beta,\gamma),(2\alpha,2\beta,2\gamma)$                          & 1,2,3 &$\mathsf{A}$\\
$E^*$    & 3 & $2xy, 2yz, 2zx$           & $(\alpha,\beta,\gamma),(\alpha+\beta,\beta+\gamma,\gamma+\alpha)$         & 1,2,3 &$\mathsf{A}$\\
$F$      & 3 & $x^2+y^2, 2xy, 2yz$       & $(\alpha,\alpha,\beta),(2\alpha,2\alpha,\alpha+\beta)$                        & 1,1 &$\neq$\\
$F^*$    & 3 & $x^2+y^2,xz,z^2$          & $(\alpha,\alpha,\beta),(2\alpha,\alpha+\beta,2\beta)$                         & 1,1 &$\Omega$\\
$G$      & 4 & $x^2,y^2,yz$              & $(\alpha,\beta,\gamma),(2\alpha,2\beta,\beta+\gamma)$                         & 1,1,1 &$\neq$\\
$G^*$    & 4 & $xy, xz,z^2$              & $(\alpha,\beta,\gamma),(\alpha+\beta,\alpha+\gamma,2\gamma)$                  & 1,1,1 &$\neq$\\
$H$      & 5 & $x^2,2xy,y^2+2xz$         & $(2\alpha,\alpha+\beta,2\beta),(4\alpha,3\alpha+\beta, 2\alpha+2\beta)$       & 1,1 &$\Xi$\\
$I$      & 7 & $x^2,xy,y^2$              & $(\alpha,\beta,\gamma), (2\alpha,\alpha+\beta,2\beta)$                        & 1,1,2 &0\\
$I^*$    & 7 & $xz,yz,z^2$               & $(\alpha,\beta,\gamma), (\alpha+\gamma,\beta+\gamma,2\gamma)$                 & 1,1,2 &0\\
\hline
$\Sigma^1$&&&&\\
\hline
$(1^4)$  & 4 & $x^2-xz,y^2-yz,0$     & $(\alpha,\alpha,\alpha), (2\alpha,2\alpha,\beta)$                         & 1,1     &\Zhe\\
$(21^2)$  & 5 & $xy,xz+yz,0$          & $(\alpha,\alpha,\beta),  (2\alpha,\alpha+\beta,\gamma)$                   & 1,1,1   &$\neq$\\
$(31)$   & 6 & $xz,x^2-yz,0$         & $(\alpha+\beta,2\alpha,2\beta), (\alpha+3\beta,2\alpha+2\beta,\gamma)$    & 1,1,1   &$\Xi$\\
$(22)$   & 6 & $x^2,yz,0$            & $(\alpha,\beta,\gamma),  (2\alpha,\beta+\gamma,\delta)$                   & 1,1,2   &$\neq$\\
$(4)$    & 7 & $xz+y^2,x^2,0$        & $(2\alpha,\alpha+\beta,2\beta),  (2\alpha+2\beta,4\alpha,\gamma)$         & 1,1,1   &$\Xi$\\
$K$      & 8 & $y^2,z^2,0$           & $(\beta,\alpha,\gamma),  (2\alpha,2\gamma,\delta)$                        & 1,1,1,2 &0\\
$L$      & 8 & $xy,xz,0$             & $(\alpha,\beta,\gamma),  (\alpha+\beta,\alpha+\gamma,\delta)$             & 1,1,1,2 &0\\
$M$      & 9 & $yz,y^2,0$            & $(\alpha,\beta,\gamma),  (\beta+\gamma,2\beta,\delta)$                    & 1,1,1,1 &0\\
\hline
$\Sigma^2$&&&&\\
\hline
$S$      & 10&  $xy-z^2, 0,0$        & $(2\alpha, 2\beta,\alpha+\beta), (2\alpha+2\beta,\gamma,\delta)$         & 1,1,2,2 &0\\
$PL$     & 11 & $xy,0,0$             & $(\alpha,\beta,\gamma), (\alpha+\beta,\delta,\epsilon)$                  & 1,1,1,2,2 &0\\
$DL$     & 13 & $x^2,0,0$            & $(\alpha,\beta,\gamma), (2\alpha,\delta,\epsilon)$                       & 1,1,1,2,2 &0\\
\hline
$0$      & 18 & $0, 0, 0$            & $(\alpha,\beta,\gamma), (\delta,\epsilon,\kappa)$                        & 1,1,2,2,3,3 &0\\
\hline
\end{tabular}\end{table}

To explain the notation consider the orbit $C$ represented by the net of conics $(y^2+2xz,2yz,-x^2)$. The pair of matrices
\begin{equation}\label{stab-C}
\left(\left(
    \begin{array}{ccc}
      a^2 & 0 & 0 \\
      0 & ab & 0 \\
      0 & 0 & b^2 \\
    \end{array}
  \right),
  \left(
    \begin{array}{ccc}
      a^2b^2 & 0 & 0 \\
      0 & ab^3 & 0 \\
      0 & 0 & a^4 \\
    \end{array}
  \right) \right)\in \gl(U)\times \gl(V),\qquad\qquad a,b\in \gl_1
  \end{equation}
stabilize this net of conic. Since the codimension of $C$ is 2, the dimension of the stabilizer subgroup has to be 2 as well ($\dim G=\dim V=18$), so we determined the maximal torus. This is the data that is encoded as $(2\alpha,\alpha+\beta,2\beta),(2\alpha+2\beta, \alpha+3\beta, 4\alpha)$ in Table~\ref{symmetries}.

\begin{theorem} \label{thm:nets_equiv_classes}
Consider the $\gl(U)\times \gl(V)$ representation $\noc$.  The Theorem of {Restriction Equations} \cite[Thm. 3.5]{cr} determines all the $\gl(U)\times \gl(V)$ equivariant classes of the codimension$>1$ orbit closures, eg. we have
  \begin{itemize}
     \item{} $[\overline{C}]=8(v_1-2u_1)^2$,
     \item{} $[\overline{D}]=-3u_2+3v_2-16u_1v_1+3v_1^2+17u_1^2$,
     \item{} $[\overline{D^*}]=12u_2-3v_2-20u_1v_1+6v_1^2+16u_1^2$,
     \item{} $[\overline{E}]=3u_3+3v_3-3u_1u_2+u_2v_1-6u_1v_1^2+13u_1^2v_1-2u_1v_2-8u_1^3+v_1^3$
     \item{} $[\overline{E^*}]=-24u_3+3v_3-24u_1u_2+16u_2v_1-16u_1v_1^2+20u_1^2v_1-6v_1v_2+10u_1v_2-8u_1^3+4v_1^3$
     \item{} $[\overline{F}]=2(v_1-2u_1)(6u_1^2-4u_1v_1-6u_2+3v_2)$,
     \item{} $[\overline{F^*}]=2(v_1-2u_1)(5u_1^2-8u_1v_1+9u_2-3v_2+3v_1^2)$.
    \end{itemize}
\end{theorem}

\begin{proof} The proof does not follow from any general principle we are aware of, it is just an experimental fact. The symmetry data of the table put constraints on the classes $[\overline{\eta}]$. One can write down all these constraints for each codimension$>1$ orbit $\eta$. A computer program shows that for each codimension$>1$ orbit there is only one equivariant class in $H^*\big(B(\gl(U)\times \gl(V))\big)$ satisfying the constraints.
\end{proof}

For the family of codimension 1 orbits we look at the Wall-DuPlessis classification \cite{wall-duplessis} from an equivariant point of view.

The affine plane $N_C=\{\nu_{c,g}: c,g\in \C\}$, where
 \[\nu_{c,g}=(y^2+2xz,2yz,-x^2+2g(xz-y^2)+cz^2),  \]
is normal to the orbit $C$ at the point $(y^2+2xz,2yz,-x^2)$. This plane is invariant under the action of the complex 2-torus $T_C$ of (\ref{stab-C}). The $T_C$ action on $N_C$ has weights $2\alpha-2\beta$ and $4\alpha-4\beta$, corresponding to the weight vectors $(0,0,xz-y^2)$ and $(0,0,z^2)$. Hence, the orbits of $T_C$ on $N_C$ correspond to the parabolas with $\mu=(c:g^2)\in \P^1$ constant.

According to \cite{wall-duplessis} these parabolas are exactly the intersections of the codimension 1 $\noc$-orbits with the normal slice $N_C$. We will refer to the orbit of $\nu_{c,g}$ with $\mu=(c:g^2)$ as $A_\mu$. In \cite{wall:nets} $A_{-9}$ is called $B$ and $A_{0}$ is called $B^*$. We will refer to a $\noc$-orbit representative lying in $T_C$ as a {\em $c$-$g$-form}. Recall the following {\em Incidence Theorem}.

\begin{theorem} \cite{incidence} \label{thm:inc=tp}  Consider a Lie group $G$ acting on a vector space $V$ complex linearly. For $v\in V$ let $G_v$ denote the stabilizer subgroup of $G$. Let $S$ be a subgroup of $G_v$ and $N_v$ an $S$-invariant normal slice to the orbit $Gv$ at $v$. Suppose that $\eta\subset V$ is a $G$-invariant subvariety. Then
  \[  [\eta\subset V]_S=[(\eta\cap N_v)\subset N_v]_S.  \]
\end{theorem}

\begin{theorem} \label{thm:codim1tp}
The equivariant classes of the $A_\mu$ orbits are $4(v_1-2u_1)$ for $\mu\neq\infty$ and $2(v_1-2u_1)$ for $\mu=\infty$.
\end{theorem}

\begin{proof} If $\mu\not=\infty$, then the class $[(\mu g^2=c)\subset N_C]$ is equal to the weight of the $g=0$ direction. Hence we have $[(\mu g^2=c)\subset N_C]=4\alpha-4\beta$. For the curve $g=0$ we have $[(g=0)\subset N_C]=2\alpha-2\beta$. For the restriction homomorphism $r:H^*_{\GL(U)\times \GL(V)}\to H^*_{T_C}=\Z[\alpha,\beta]$ we have $r(u_1)=2\alpha+(\alpha+\beta)+2\beta=3\alpha+3\beta$ and $r(v_1)=(2\alpha+2\beta)+(\alpha+3\beta)+4\alpha=7\alpha+5\beta$ by (\ref{stab-C}).
Hence, if $[A_\mu]=Au_1+Bv_1$ ($\mu\not=\infty$), then according to Theorem \ref{thm:inc=tp}, we have $A(3\alpha+3\beta)+B(7\alpha+5\beta)=4\alpha-4\beta$. The only solution is $A=-8, B=4$. For $\mu=\infty$ the calculation is similar.
\end{proof}

\begin{remark} From the equivariant classes of (cone) varieties one can calculate their degrees, see eg. \cite[Sec. 6]{forms} and Section \ref{sec:mult}. Therefore, Theorem \ref{thm:codim1tp} implies that the degree of the hypersurfaces given by the closures of the $A_\mu$ orbits is $4(1+1+1-2(0+0+0))=12$ in general, and 6 for $\mu=\infty$. This implies that the ring of invariants $R(\noc)$ on $\noc$ is generated by a degree 6 and a degree 12 polynomial. We will give a full description of $R(\noc)$ in Section \ref{sec:inv-of-noc}.
\end{remark}

\section{Thom polynomials of contact singularities corresponding to  nets of conics} \label{sec:ts}

Consider a polynomial map  $g : (\C^n, 0) \to (\C^p, 0)$. In global singularity theory one studies its Thom polynomial
$\Tp_g\in \Z[\alpha_1,\ldots,\alpha_n,\beta_1,\ldots,\beta_p]^{S_n\times S_p}$. Here $S_n$ permutes the $\alpha_i$ variables (the so-called ``source Chern roots"), and $S_p$ permutes the $\beta_j$ variables (the so-called ``target Chern roots'').

For the precise definition and properties of $\Tp_g$ see for example \cite{ts_loc} and references therein. Let us just recall one property of Thom polynomials. For $g$ above define $I_g$ to be the ideal in $\C[x_1,\ldots,x_n]$ generated by the coordinate functions $g_1,\ldots,g_p$ of $g$. A property of $\Tp_g$ is that it only depends on the isomorphism type of the local algebra $\C[x_1,\ldots,x_n]/I_g$. In particular, if for $g_1$ and $g_2$ we have $I_{g_1}=I_{g_2}$ then $\Tp_{g_1}=\Tp_{g_2}$.

\smallskip

Consider a net of conics, that is, a point $f\in\Hom(S^2U,V)$. This can be represented by three degree-2 polynomials $f_1,f_2,f_3$ in $x,y,z$ (see examples in Table \ref{symmetries}). We can assign to $f$ the ideals $I_f=(f_1,f_2,f_3)$ and $I_{\tilde f}=(f_1,f_2,f_3)+(x,y,z)^3$ in $\C[x,y,z]$. The first choice seems more natural, however the codimension of $I_f$ in $\C[x,y,z]$ depends on $f$. On the other hand $I_{\tilde f}$ has codimension $7$ independently of $f\in \Sigma^0$. This allows us to use the localization formula of \cite[Thm 6.1]{ts_loc} (see also \cite[\S9 and 12]{ts_loc}).

\begin{theorem}\label{main}
Let $f\in \Hom(S^2U,V)$ be a $\Sigma^0$ net of conics, and let $[f]\in\Z[u_1,u_2,u_3,v_1,v_2,v_3]$ be the equivariant class of its $\GL(U)\times \GL(V)$-orbit considered in Section \ref{sec:class}. Let $\sigma_f:(\C^3,0)\to (\C^p,0)$ be a polynomial map with $I_{\tilde f}=I_{\sigma_f}$. Let $W=\{2\alpha_1,2\alpha_2,2\alpha_3,\alpha_1+\alpha_2,\alpha_1+\alpha_3,\alpha_2+\alpha_3\}$.
Then
\begin{equation}\label{eq:loc}
    \Tp_{\sigma_f}=\prod_{i=1}^3\prod_{j=1}^p (\beta_j-\alpha_i)\cdot \sum_{H\subset W, |H|=3}\frac{D_H \cdot [f]|_H}{e_H},
  \end{equation}

where
\[D_H=\prod_{w\notin H}\prod_{j=1}^p (\beta_j-w), \qquad e_H=\prod_{w_1\notin H}\prod_{w_2\in H} (w_2-w_1),
\qquad
[f]|_H=\left.[f]\right|_{v_i=\sigma_i(H), u_i=\sigma_i(\alpha_1, \alpha_2,\alpha_3)}.
\]
In the last substitution $\sigma_i$ means the $i$'th elementary symmetric polynomial.
\end{theorem}

\begin{remark}\label{degenerate}
Similar, and in some sense simpler formulas can be given for the non-$\Sigma^0$ nets. The corresponding polynomial maps belong to very large codimensional singularities, consequently their Thom polynomials are of very large degree, so we omit the description.
\end{remark}

Using the terminology of \cite{ts_loc} Theorem \ref{main} gives the Thom polynomial in the Chern root variables $\alpha_i$ and $\beta_j$. A more compact way to write down these polynomials is in {\em quotient variables}, using the basis of Schur polynomials. To indicate that we use quotient variables we use  the notation $\tp_{\sigma_f}$. The Schur polynomials $\Delta$ are defined by $\Delta_{(\lambda_1,\ldots,\lambda_r)}=\det(c_{\lambda_i+j-i})_{r\times r}$, for example
$\Delta_{(31)}=\det\begin{pmatrix} c_3 & c_4 \\ c_0 & c_1 \end{pmatrix}=c_3c_1-c_4$, and $c_i$ denotes the degree $i$ homogeneous part of $\frac{\prod_{j=1}^p (1+\beta_j)}{\prod_{i=1}^3 (1+\alpha_i)}$.

\smallskip

For example Theorem \ref{main}, and some calculation, gives that for $\sigma_f:\C^3\to \C^4$,
\[ \sigma_f(x,y,z)=(y^2+2xz,2yz,-x^2+2g(xz-y^2)+cz^2,z^3)\]
we have
\begin{equation} \label{eq:qwqw}
\tp_{\sigma_f}=8\Delta_{(544111)}+4\Delta_{(444211)}+16\Delta_{(844)}+20\Delta_{(6442)}+32\Delta_{(64411)}
      +120\Delta_{(6541)}+160\Delta_{(655)}+
\end{equation}
$$\phantom{mmm}+16\Delta_{(54421)}+32\Delta_{(55411)}+40\Delta_{(5542)}+80\Delta_{(5551)}+80\Delta_{(664)}+40\Delta_{(7441)}
+112\Delta_{(754)}$$
if $g\neq0$ and half of it for $g=0$.

\subsection{Thom polynomials of equidimensional maps}
Notice that for any $\sigma_f$ satisfying the condition of Theorem \ref{main}, $p$ is at least 4. It is possible to substitute however $p=3$ formally into the formula, and in some cases we can interpret the result as a Thom polynomial. This phenomenon was called the "small $p$ case" in \cite[\S 12]{ts_loc}. Applying the argument to our case we get

\begin{proposition}\label{eqidim}
Let $f$ be a $\Sigma^0$ net of conics. Interpret $f$ as a polynomial map $f:(\C^3,0)\to (\C^3,0)$ as well. Expression (\ref{eq:loc}) is the Thom polynomial of this polynomial map, if the following two conditions are satisfied.
\begin{enumerate}
  \item The only ideal $I$ containing $I_f$ with $\C[x,y,z]/I\iso \C[x,y,z]/I_{\tilde f}$ is $I_{\tilde f}$.
  \item The codimension of the polynomial map $f:(\C^3,0)\to (\C^3,0)$ in the singularity sense is $\deg[f]+9$.
\end{enumerate}
\end{proposition}

It is easy to see that the first condition is valid for all $f\in\Sigma^0$ (Any such $I$ should contain the ideal $(x,y,z)^3$). We give 3 examples when the second condition is also satisfied (see \cite[p.~315]{wall-duplessis}).

1. The Thom polynomial of the germ $h_{c,g}:\C^3\to \C^3$,
      \[h_{c,g}(x,y,z)=(y^2+2xz,2yz,-x^2+2g(xz-y^2)+cz^2)\]
      is
    \[ \tp_{h_{c,g}}=8\Delta_{(433)}+4\Delta_{(3331)},\]
if $g\neq 0$ and half of it for $g=0$. The singularities $h_{c,g}$ form the smallest codimensional example of a family of non-equivalent contact singularities $(\C^n,0)\to (\C^n,0)$ (\cite{mather6}). Because of the presence of this continuous modulus in the classification of singularities, the method of \cite{rrtp} calculating Thom polynomials broke down at codimension 9. (In \cite{ts_loc} a different representation of $h_{c,g}$ is used, namely $(x^2-\lambda yz, y^2-\lambda xz, z^2-\lambda xy)$. This representation is less adapted to our purposes---in general 12 different values of $\lambda$ correspond to the same orbit, and the singularities corresponding to $B$ and $B^*$ cannot be written in this form.)

2 and 3: The singularities corresponding to nets from the orbits $D$ and $E$ are denoted by $KD$ and $KE$ in \cite[p. 315]{wall-duplessis}.  Hence

\begin{theorem}\label{DE}
The Thom polynomials of maps $(\C^n,0) \to (\C^n,0)$ of type $KD$ and $KE$ are
\[  \tp_{KD}=3\Delta_{(33311)}+6\Delta_{(3332)}+14\Delta_{(443)}+16\Delta_{(4331)}+17\Delta_{(533)},\]
\[  \tp_{KE}=\Delta_{(333111)}+2\Delta_{(33321)}+4\Delta_{(3333)}+6\Delta_{(43311)}+8\Delta_{(4332)}+\]
\[\phantom{mmmmm}+14\Delta_{(4431)}+8\Delta_{(444)}+13\Delta_{(5331)}+19\Delta_{(543)}+8\Delta_{(633)}.\]
\end{theorem}

\begin{remark} A polynomial map is called equidimensional, if the dimension of its source is the same as the dimension of its target. Thom polynomials of
equidimensional polynomial maps were calculated in \cite{rrtp} up to degree 8. The results above on the Thom polynomials of $h_{c,g}$, $KD$, and $KE$ are essentially the first examples beyond degree 8.
\end{remark}

\subsection{Duality and the Kazarian theory}
Recently M. Kazarian developed a method (\cite{kaza:noas}) to calculate Thom polynomials. His approach is dual to the method of \cite{ts_loc} in the sense that it uses the quotient algebra as opposed to the ideal of the singularity. For a singularity corresponding to a $\Sigma^0$ net $f$ in the above sense the quotient algebra is $Q_f=J^2(U,V)/(f_1,f_2,f_3)$. This is a graded algebra of graded degree $(3,3)$, so the multiplication is determined by a linear map $m\in \Hom(S^2U^*,\ker f)$. The map $m$ is the {\em apolar} of $f$ in the terminology of \cite{wall:nets}. The input of the Kazarian method is the equivariant class of $m$. The apolar of a net lives in the dual representation, however, by fixing a basis of $U$ we can identify these spaces. The identification depends on the basis but a basis change does not change the orbit, so we get a well defined duality on the $\Sigma^0$ orbits. Wall denotes this duality by $*$. It is also true (and easy to see) that the equivariant classes of the orbits are the same up to sign in the dual representation, however in general the dual orbit has different equivariant class, see e.g. $[D]\neq [D^*]$ and consequently the Thom polynomials of the corresponding singularities are also different. Using the Kazarian theory we were able to double check our results.

\section{Invariant theory of nets of conics}\label{sec:inv-of-noc}
Most of the results in this section are known (see \cite{wall:nets}, \cite{wall:nocinv}, \cite{wall-duplessis} \cite{vinberg} and \cite{artamkin-nurmiev}) ), but we also would like to show how equivariant theory leads to these results with the hope that it can be applied in a more general context.

\subsection{The ring of (semi) invariants}
Suppose that the Lie group $G$ acts on the vector space $W$ and $\hat{G}$ is the character group of $G$. We say that $f\in \C[W]$ is a {\em relative invariant corresponding to the character} $\chi\in\hat{G}$ (i.e. $f\in R_\chi(W)$) if for all $v\in W$ and $g\in G$
\begin{equation}\label{semi} f(gv)=\chi(g)f(v). \end{equation}
The ring of {\em semi-invariants} is $R(W):=\bigoplus_{\chi\in \hat{G}} R_\chi(W)$.

Note that an element of $R(W)$ is not necessarily a relative invariant for any $\chi\in \hat{G}$. Semi-invariants of $G$ are always invariants of the commutator subgroup $G'$, but in general the ring of invariants of $G'$ can be bigger.  In the following two examples, however, they coincide.

\begin{enumerate}[(a)]
\item The $\gl(U)\times\gl(V)$-action on $\noc$: Any character of $\gl(U)\times\gl(V)$ is of the form
$\chi_{a,b}(g,h):=\det^a(g)\det^b(h)$. If $f\in R(\noc)$ is homogeneous of degree $l=3d$, then
\[ f\big((\lambda I,\mu I)v\big)=f(\lambda^{-2}\mu v)=\lambda^{-2l}\mu^l f(v)={\det}^{-2d}(\lambda I){\det}^d(\mu I)f(v), \]
therefore $f$ is a relative invariant corresponding to the character $\chi_{-2d,d}$. (In GIT language we see that there is a unique linearization for GIT-quotient.) In other words
\[ R(\noc)=\bigoplus_{d\in \N} R_{\chi_{-2d,d}}(\noc). \]
In this case all characters are determinants, so $R(\noc)$ coincides with the ring of absolute invariant polynomials for the $\SL_3\times\SL_3$-action on $\noc$.

\item As we discussed before, the maximal torus of the stabilizer of the net $(y^2+2xz,2yz,-x^2)$ is $\gl_1\times\gl_1$. It acts on the normal space $N_C$ to its orbit with weights $4\alpha-4\beta$ and  $2\alpha-2\beta$. With respect to this action we have
\[ R(N_C)=\bigoplus_{d\in \N} R_{\chi_{2d,-2d}}(N_C)=\C[N_C]=\C[c,g], \]
where $c\in R_{\chi_{4,-4}}(N_C)$ and $g\in R_{\chi_{2,-2}}(N_C)$. Consider the restriction map $i^*:R(\noc)\to R(N_C)$. From the first line of Table \ref{symmetries} (equivalently, from (\ref{stab-C})) one sees that
$i^*$ maps $R_{\chi_{-2d,d}}(\noc)$ into $R_{\chi_{d,-d}}(N_C)$.
\end{enumerate}

We claim that $i^*:R(\noc)\to R(N_C)$ is injective. Indeed, according to the splitting above it is enough to show that it is injective on  $R_{\chi_{-2d,d}}(\noc)$ for any given $d$, which follows from the fact that  the values of a relative invariant $f$ on $N_C$ determine its values on $(\gl(U)\times\gl(V)) \cdot N_C$ via (\ref{semi}) and $(\gl(U)\times\gl(V)) \cdot N_C$ is dense in $\noc$.  Now we prove that the homomorphism $i^*$ is also surjective, by finding relative invariants of $\noc$ mapped to $c$ and $g$.

\subsection{The determinant map} \label{sec:determinant} Composing nets with the determinant $S^2U^{*}\to \C$ (well defined upto a scalar factor) we get a degree 3 polynomial map $\delta:\noc\to S^3V$. This map is $\gl(U)\times\gl(V)$-equivariant if we let  $\gl(U)$ act on $S^3V$ as scalars by the second tensor power of the determinant representation. Consequently we have a homomorphism $\delta^*:R(S^3V)\to R(\noc)$. Let us review now the invariant theory of the plane cubics $S^3\C^3$ which was one of the first achievements of the early invariant theory.

\begin{theorem}\cite{Aronhold} \label{thm:cubics-invs}  There are invariants $a,b$ of $S^3\C^3$ of degree $4$ and $6$, respectively such that $R(S^3\C^3)\iso\C[a,b]$ and every smooth plane cubic $\gamma$ can be transformed (using the $\GL_3$-action) into the Weierstrass-form:
\[y^2z+x^3+a(\gamma)xz^2+b(\gamma)z^3.\]
\end{theorem}

\begin{remark} The Weierstrass-form is analogous to the $c$-$g$-form of nets of conics from Section~\ref{sec:class}. The subset $\{y^2z+x^3+axz^2+bz^3:a,b\in \C\}$ is a normal slice to the orbit of the cuspidal cubic $y^2z+x^3$. This observation was used in \cite{komuves:thesis} to calculate the equivariant classes for plane cubics. The choice of these orbits is not accidental. Their closure is the nullcone, so the normal slice intersects all invariant hypersurfaces.\end{remark}

Consider now the determinant of the $c$-$g$-form $\nu_{c,g}=(y^2+2xz,2yz,-x^2+2g(xz-y^2)+cz^2)$. Considering the three $3\times 3$ matrices of the three components of $\nu_{c,g}$ we obtain
\begin{align*}
\delta(\nu_{c,g}) & =  \det\left(
   (-x) \left(
   \begin{array}{ccc}
      0 & 0 & 1 \\
      0 & 1 & 0 \\
      1 & 0 & 0 \\
    \end{array} \right) +
    (-y) \left(
    \begin{array}{ccc}
      0 & 0 & 0 \\
      0 & 0 & 1 \\
      0 & 1 & 0 \\
    \end{array}
  \right) +
    z \left(
    \begin{array}{ccc}
      -1 & 0 & g \\
      0 & -2g & 0 \\
      g & 0 & c \\
    \end{array}
  \right)\right) \\
    & =  y^2z+x^3+(c-3g^2)xz^2+2g(c+g^2)z^3.
  \end{align*}
Here we parameterized $\C^3$ with $-x,-y,z$ to obtain our result, the Weierstrass form, without sign changes. Therefore $i^*\delta^*a=c-3g^2$ and $i^*\delta^*b=2g(c+g^2)$. We denote the degree 12 invariant $-48\delta^*a$ by $J_{12}$.

To complete the calculation of $R(\noc)$ we need to find the degree 6 invariant which restricts to $g$ on the normal slice $N_C$. This is a straightforward job with a computer, but using a geometric idea it can be done by hand.

\subsection{The Pl\"ucker map and the invariant $I_2$} We have a degree 3 $\gl(U)\times\gl(V)$-equivariant map
\[\psi:\noc=\Hom(S^2U,V)\ \stackrel{\bigw^3}\to\ \Hom(\bigw^3S^2U,\bigw^3V)\iso \bigw^3 S^2U^*\otimes \bigw^3V.\]
We will call $\pl:=\bigw^3 S^2U^*\otimes \bigw^3V$ the {\em Pl\"ucker space}.

Picking a basis for  $V^*$, to each net we can associate a triple $M=(M_1,M_2,M_3)$ of conics, and $\psi$ sends $M$ to $M_1\wedge M_2\wedge M_3\in \bigwedge^3 S^2U^*$. The image of $\psi$ is the cone of the Grassmannian $\gr_3(S^2U^*)$. Our next goal is to show that the representation $\pl$ has a degree 2 invariant $I_2$ which pulls back to a degree 6 invariant $J_6$ of $\noc$.

Let $\xi,\eta,\nu$ be a basis in $U^*$ and $x,y,z$ the corresponding dual basis in $U$. Write $E_{ij}$ for the element of $\mathfrak{sl}(U)$ mapping the $i$th basis vector of $U$  to the $j$th basis vector  and annihilating the other two basis vectors.

Consider the basis $e_1:=\xi^2$, $e_2:=\xi\eta$, $e_3:=\xi\nu$, $e_4:=\eta^2$,  $e_5:=\eta\nu$, $e_6:=\nu^2$ in $S^2U^*$, and set $e_{ijk}:=e_i\wedge e_2\wedge e_3\in \bigwedge^3 S^2(U^*)$.
Under the natural identification of  $S^2U$ with the dual of $S^2(U^\star)$, the basis in $S^2(U)$ dual to $e_1,\dots,e_6$ is $t_1:=x^2$, $t_2:=2 xy$, $t_3:=2 xz$, $t_4:=y^2$, $t_5:= 2 yz$, $t_6:=z^2$.
Then  $t_{ijk}:=t_i\wedge t_j\wedge t_k\in\bigwedge^3S^2(U)$ is the basis dual to $e_{ijk}$. In particular,
$t_{ijk}(\pi(M))$ is the $3\times 3$ minor corresponding to the $(i,j,k)$ columns of the $3\times 6$ matrix
of the net $M$ viewed as a linear map from $V^*$ to $S^2U^*$, with respect to the chosen bases.

As an $\SL(U)$-representation $\pl\iso S^3(U)\oplus S^3(U^*)$ (this follows for example from the calculations below or by calculating the weights of $\pl$ (see Figure \ref{fig:weights}).).

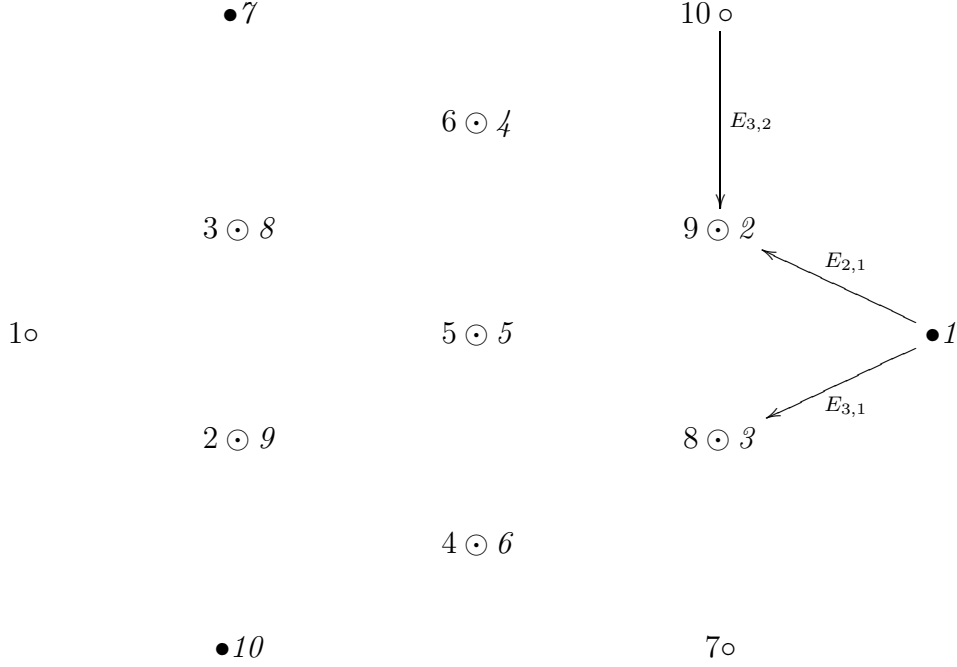
\begin{figure}
\xymatrix{ &&\bullet\mathit7 &&              &&10\circ\phantom{\mathit2}\ar[dd]^{E _{3,2}}       && \\
           &&                &&6\odot\mathit4 &&               && \\
           &&3\odot\mathit8  &&              &&9\odot\mathit2 && \\
   1 \circ &&                &&5\odot\mathit5 &&      &&\bullet\mathit1 \ar[ull]_{E _{2,1}} \ar[dll]^{E _{3,1}}\\
           &&2\odot\mathit9  &&              &&8\odot\mathit3 && \\
           &&      &&4\odot\mathit6 &&      && \\
           &&\bullet\mathit{10} &&      &&7\circ &&  }
  \caption{The 20 weights of $\pl$}\label{fig:weights}
\end{figure}

Denoting by  $W$ the $10$-dimensional $\SL(U)$-module $S^3(U)$ we have $S^2 \pl^*\cong S^2W\oplus S^2W^*\oplus W\otimes W^*$. The first two summands do not contain the trivial $\SL(U)$-module (say by the theorem of Aronhold on the invariants of ternary cubic forms), and the third summand contains one copy of the trivial representation by Schur's Lemma, spanned by $w_1w_1^*+\cdots+w_{10}w_{10}^*$, where $w_1,\dots,w_{10}$ is a basis of $W$ and $w_1^*,\dots,w_{10}^*$ is the corresponding dual basis in $W^*$. Table \ref{un-normalized} contains explicit elements $w_1,\dots,w_{10}\in \bigwedge^3 S^2(U)$ spanning an $\mathfrak{sl}(U)$-summand isomorphic to $W$. Each weight space in $\bigwedge^3 S^2U$ is $1$-dimensional. Therefore $x^2\wedge xy\wedge xz$ is a highest weight vector generating an $\mathfrak{sl}(U)$-module isomorphic to $W$. The $w_2,\dots,w_{10}$ are obtained by applying successively the operators $E_{21}, E_{31}\in\mathfrak{sl}(U)$, as indicated in Table \ref{un-normalized}.

Recall that the action of $E_{i,j}$ is the sum of the replacements of each occurrence of the $j$th basis vector to the $i$th one, e.g.
      \[E_{3,1}(x^2\wedge xy\wedge xz)=2xz\wedge xy\wedge xz+x^2\wedge yz\wedge xz+x^2\wedge xy\wedge z^2=
                                               x^2\wedge xy\wedge z^2-x^2\wedge xz\wedge yz.\]

Up to non-zero scalars $w_{10}^*$ in Table \ref{un-normalized} is the only weight vector whose weight is opposite to the weight of the lowest weight element $w_{10}$ in $W$, therefore it must be a highest weight vector generating a submodule isomorphic to $W^*$ in $\bigwedge^3 S^2U^*$. Applying successively appropriate elements of $\mathfrak{sl}(U)$ to $w_{10}^*$ one computes $w^*_2,\dots,w^*_{10}\in \bigwedge^3 S^2U^*$. For example, by the left column of the table we have $w_{10}=E_{31}w_6$, hence $w_6^*=-E_{31}w^*_{10}$.
In addition to the information in the left column of the table we need also the relations
   \[  w_8=\frac 13 E_{32}w_7, \quad w_9=\frac 12 E_{32} w_8, \quad\mbox{and} \quad  w_{10}=E_{32} w_9.  \]
With the $w_i,w^*_i$ given in Table \ref{un-normalized} (for example, $w_2=-x^2 \wedge xz\wedge  y^2+ x^2\wedge xy\wedge yz=-\frac 12t_{134}+\frac 14 t_{125}$)
 we have the equality
\begin{equation} \label{eq:i2}
\begin{split}
-8I_2:=8\sum_{i=1}^{10 }w_iw_i^*&=t_{235}^2-8t_{146}^2 \\
       &-8t_{134}t_{346}+8t_{126}t_{246}+8t_{145}t_{156} \\
       &+6t_{123}t_{456}-6t_{136}t_{245}+6t_{124}t_{356}\\
       &-4t_{125}t_{256}+4t_{135}t_{345}-4t_{234}t_{236} \\
       &+2t_{134}t_{256}-2t_{125}t_{346}+2t_{135}t_{246}-2t_{126}t_{345}\\
       &+2t_{145}t_{236}+2t_{156}t_{234}-2t_{146}t_{235}.
\end{split}
\end{equation}

\begin{remark}  (i) Set $u_{ijk}:=t_{ijk}\circ\psi$, so $u_{ijk}$ is an element of the coordinate ring of $\noc$. Recall a classical result of Sylvester (see page 365 in \cite{salmon} or \cite{gizatullin}), asserting (after a change to our coordinate system) that
\begin{equation*}
\begin{split}
-8\theta:=&\, u_{235}^2-8u_{146}^2 \\
         +&\, 4u_{146}u_{235}+4u_{135}u_{345}-4u_{125}u_{256}-4u_{234}u_{236}\\
         +&\, 8u_{145}u_{156}-8u_{134}u_{346}+8u_{126}u_{246} \\
         +&\, 8u_{123}u_{456}-8u_{136}u_{245}+8u_{124}u_{356}
\end{split}
\end{equation*}
is an $\SL(U)$-invariant on $\noc$. We thank I. Dolgachev for bringing this reference to our attention.
One can easily verify using the straightening algorithm (cf. Section 13.2.2 in \cite{procesi}) that
$J_6:=I_2\circ\psi$ coincides with Salmon's $\theta$. Notice that one cannot reconstruct $I_2$ from $\theta$ since $\psi^*$ has a kernel, generated by the {\em Plücker relations}. To the best of our knowledge formula (\ref{eq:i2}) for $I_2$ is new.

(ii) It is proved in \cite{vinberg} that the ring of $\SL(\mathbb{C}^3)\times \SL(\mathbb{C}^3)\times \SL(\mathbb{C}^3)$-invariants in
$\mathbb{C}^3\otimes\mathbb{C}^3\otimes\mathbb{C}^3$ is generated by three algebraically independent elements of degree $6$, $9$, and $12$.
An alternative way to construct  the invariant $J_6$ (not as a pullback from $\pl$) is to restrict the degree $6$ generator to the subspace of $\C^3\otimes\C^3\otimes \C^3$ corresponding  to symmetric matrix triples. Explicit formulae for the three generators can be found in \cite{thibon} or \cite{matyi}.
\end{remark}

One can calculate that
\[\psi(\nu_{c,g})=(y^2+2xz)\wedge2yz\wedge(x^2+2g(xz-y^2)+cz^2)=12ge_{345}+2c(e_{456}+2e_{356})-2e_{145}-4e_{135},  \]
and hence, for $J_6=I_2 \circ \psi$ we have
$J_6(\nu_{c,g})=I_2(\psi(\nu_{c,g}))=24g$. This concludes our proof of the following known theorem (see \cite{artamkin-nurmiev}, \cite{vinberg}, where the ring of invariants of $\noc$ is identified with the ring of invariants of a finite complex pseudo-reflection group):

\begin{theorem}\label{thm:noc-inv-ring}  The ring of invariants of $\noc$ is freely generated by $J_6$ and $J_{12}$. \end{theorem}

\begin{table}\caption{Un-normalized generators}\label{un-normalized}\begin{tabular}{|l@{ }p{1mm}l|l@{ }p{1mm}l|}

\hline
$w_1$&=& $x^2\wedge xy\wedge xz$           &
                                      $  w^*_{10}$ &=&$x^2\wedge xy\wedge y^2$
\\  \hline
$w_2 =E_{21}w_1$&=&$-x^2\wedge xz \wedge y^2+ x^2\wedge xy\wedge yz$   &

                                      $w^*_9 =-E_{32}w^*_{10} $ &=& $-x^2\wedge xz\wedge y^2-2x^2\wedge xy\wedge yz $
\\  \hline
$w_3 =E_{31}w_1$&=&$x^2\wedge  yz\wedge xz+ x^2\wedge xy\wedge z^2$&
                                      $w^*_8 =-\frac 12 E_{32}w^*_9$&=&$2x^2\wedge xz\wedge yz+ x^2\wedge xy\wedge z^2$
\\  \hline
$w_4 =E_{21}w_2$&=& $2xy\wedge y^2\wedge xz+ 2x^2\wedge y^2\wedge yz$&
                                      $  w^*_7 =-\frac 13 E_{32}w^*_8$&=&$-x^2 \wedge  xz\wedge z^2$
\\  \hline
$w_5 =E_{31}w_2$&=&$x^2\wedge y^2\wedge z^2+2xz\wedge xy\wedge yz$&
                                      $w^*_6 =-E_{31}w^*_{10}$&=&$2xy\wedge xz\wedge y^2+x^2\wedge y^2\wedge yz$
\\  \hline
$w_6=E_{31}w_3$&=&$2x^2\wedge yz\wedge z^2+ 2xz\wedge xy\wedge z^2$&
                                      $w^*_5=-E_{31}w^*_9$&=&$-4xy\wedge xz\wedge yz-x^2\wedge y^2\wedge z^2$
\\  \hline
$w_7 =E_{21}w_4$&=&$6xy \wedge  y^2\wedge yz$                  &
                                      $w^*_4 =-E_{21}w^*_7$&=&$2xy\wedge xz\wedge z^2+x^2\wedge yz\wedge z^2$
\\  \hline
$w_8=E_{31}w_4$&=&$2xy\wedge y^2\wedge z^2+2xz\wedge y^2\wedge yz$&
                                      $w^*_3=-E_{31}w^*_6$&=&$2xy\wedge y^2\wedge z^2-4xz\wedge y^2\wedge yz$
\\  \hline
$w_9=E_{31}w_5$&=&$2xy\wedge yz\wedge z^2+2xz\wedge y^2\wedge z^2$&
                                      $w^*_2=-E_{21}w^*_4$&=&$2xz\wedge y^2\wedge z^2-4 xy\wedge yz\wedge z^2$
\\  \hline
$w_{10} =E_{31}w_6$&=& $6xz\wedge  yz\wedge z^2$&
                                      $  w^*_1 =-E_{21}w^*_2$&=& $6y^2\wedge yz\wedge z^2$
\\  \hline
\end{tabular}\end{table}

\subsection{A geometric interpretation of the splitting of $\pl$}
Following  C.T.C. Wall \cite{wall:nocinv} we can interpret the projection maps from $\pl$ to its irreducible factors.

\subsubsection{The Jacobi map}A net $\varphi$ is a linear map from $S^2U$ to $V$, alternatively a quadratic map from $U$ to $V$. Its derivative at $u\in U$ is a linear map $d_u\varphi:T_u U\to T_u V$. Since tangent spaces of a vector space can be canonically identified with the vector space itself and $d_u\varphi$ is linear in $u$, the derivative $d\varphi$ defines a linear map from $U$ to $\Hom(U,V)$. We also have the degree 3 determinant map
  \[\det:\Hom(U,V)\ \stackrel{\bigw^3}\to\ \Hom(\bigw^3U,\bigw^3V)\iso \bigw^3 U^*\otimes \bigw^3V,\]
which can be composed with $d\varphi$ to obtain a degree 3 map $\jac(\varphi)$ from $U$ to $\bigw^3 U^*\otimes \bigw^3V$. We can also consider $\jac$ as a map
\[  \jac: \noc \to S^3U^*\otimes \bigw^3 U^*\otimes \bigw^3V.\]
The map $\jac$ factors through the Pl\"ucker map providing a linear projection
\[ \pi_1:\pl\to S^3U^*\otimes \bigw^3 U^*\otimes \bigw^3V.\]
Picking a basis in $V^*$ we can identify $\jac$ with
the Jacobian covariant $\jac:\bigoplus^3S^2(U^*)\to S^3(U^*)$, which is a joint covariant of triples of conics defined as

\[ \jac(M):=\det \left(\begin{array}{ccc}\partial_{\xi} M_1 & \partial_{\xi}M_2 & \partial_{\xi}M_3 \\\partial_{\eta}M_1 & \partial_{\eta}M_2   & \partial_{\eta}M_3 \\\partial_{\nu}M_1 & \partial_{\nu}M_2 & \partial_{\nu}M_3\end{array}\right)  \]
(recall that  $\xi,\nu,\eta$ is our basis in $U^*$, and the $M_i$ are homogeneous quadratic polynomials in $\xi,\eta,\nu$). Now  $\jac$ is an alternating trilinear function in $M_1,M_2,M_3$, hence it factors through an $\SL(U)$-equivariant linear map
\[\pi_1:\pl\to S^3(U^*).\]
 It maps the basis vector $e_{ijk}\in \pl$ to the $3\times 3$ minor corresponding to the $i,j,k$ columns of the matrix
\[\left(
  \begin{array}{cccccc}
    2\xi & \eta & \nu & 0 & 0 & 0 \\
    0 & \xi & 0 & 2\eta & \nu & 0 \\
    0 & 0 & \xi & 0 & \eta & 2\nu \\
  \end{array}
\right)\]
(the columns contain the partial derivatives for each of   $\xi^2,\xi\eta,\xi\nu,\nu^2,\eta\nu,\nu^2$). Recall that $x^3,3x^2y$, $\dots$, $6xyz,\dots$ is the basis in $S^3(U)$ dual to the basis $\xi^3,\xi^2\eta,\dots,\xi\eta\nu,\dots$ of $S^3(U^*)$. Now $\pi_1^*$ embeds $S^3(U^*)^*\cong S^3(U)$ into $\pl^*$. For example, $\pi_1(e_{123})=2\xi^3$, and no other $\pi_1(e_{ijk})$ contains the monomial $\xi^3$.
This shows that $\pi_1^*(x^3)=2t_{123}=8w_1$ (where $w_1$ is the element given in Table~\ref{un-normalized}).
Similarly $\xi^2\eta$ is contained only in $\pi_1(e_{125})=2\xi^2\eta$ and $\pi_1(e_{134})=-4\xi^2\eta$.
It follows that  $\pi_1^*(3x^2y)=2t_{125}-4t_{134}=8w_2$. One checks in the same way that the basis
$x^3,3x^2y,\dots 6xyz,\dots,z^3$ of $S^3(U)$  is mapped under $\pi_1^*$ onto $8w_1,\dots,8w_{10}$ (cf. Table~\ref{un-normalized}).

\subsubsection{The dual Jacobi map}  We have a degree 2 map
\[\Hom(\C^2,U)\ \stackrel{S^2}\to\ \Hom(S^2\C^2,S^2U).  \]
Composing with a net $\varphi$ and choosing an element in $\Hom(\C^2,U)$ we get a linear map in $\Hom(S^2\C^2,V)$. Taking its determinant we get a degree 6 map from $U\oplus U\iso\Hom(\C^2,U)$ to the one-dimensional vector space $L=\Hom(\bigwedge^3S^2\C^2,\bigwedge^3V)\iso \bigwedge^3V$. (As a representation of $\GL(U)$ the line $L$ is isomorphic to the trivial one-dimensional representation.) Notice that this map factors through $\bigwedge^2 U$, providing a degree 3 map from $\bigwedge^2 U$ to $L$. Varying $\varphi$ we end up with a degree 3 map from $\noc \to S^3\bigwedge^2 U^*\otimes L$ which factors through the Pl\"ucker map $\psi$. Now notice that $\bigwedge^2 U^*\iso U\otimes \bigw^3 U^*$. Hence we defined a linear map
\[ \pi_2:\pl\to S^3U\otimes \left(\bigw^3 U^*\right)^3\otimes \bigw^3 V. \]
Notice that the $\GL(V)$-action played no active role in the projections $\pi_1$ and $\pi_2$ as it was expected from the abstract splitting of the representation $\bigw^3 S^2U^*$.

As an $\SL(U)$-equivariant linear map
$\pi_2:\pl\to S^3(U)$ can be constructed as follows: Picking a basis in $V^*$ we can identify a net with a triple $M=(M_1,M_2,M_3)$ where $M_i\in S^2U^*$. We may think of $S^3(U)$ as the space of cubic polynomial functions on $U^*$. Now $\pi_2(M_1\wedge M_2\wedge M_3)$ vanishes on a linear form $f\in U^*$ if the net $M$ restricted to the zero locus of $f$ does not have full rank.

More explicitly, eliminate the variable $\nu$ from the ternary quadratic forms  $M_i$ using the relation $x\xi+y\eta+z\nu=0$;  we obtain three binary quadratic forms in the variables $\xi,\eta$. Now $\pi_2(M_1\wedge M_2\wedge M_3)$ is  the determinant of the $3\times 3$ matrix whose columns contain the coefficients of these three binary quadratic forms. In particular, $\pi_2(e_{ijk})$ is the $(i,j,k)$ minor of
\[ \left(\begin{array}{cccccc}z & 0 & -x & 0 & 0 & x^2/z \\0 & 0 & 0 & z & -y & y^2/z \\0 & z & -y & 0 & -x & 2xy/z\end{array} \right) \]
(showing also that we end up with a cubic polynomial in $x,y,z$).
The dual $\pi_2^*$ embeds $S^3U^*$ into $\pl^*$, and in the same way as in the case of the Jacobi map one may check that the basis vectors $\xi^3,\xi^2\eta,\dots,\nu^3$ are mapped to  $\frac{-1}{3}w_1^*,\dots,\frac{-1}{3}w_{10}^*$ from Table~\ref{un-normalized}.

\subsection{Stability} \label{sec:stability}
A net of conics is in the nullcone if both $J_6$ and $J_{12}$ are zero on it. The geometric quotient of $\noc$ is $\P^1$ and the quotient map on the complement of the nullcone is given by $k:=J_6^2/J_{12}$. An orbit $\eta$ is strictly semistable if $k^{-1}(k(\eta)) \supsetneqq \eta$. We can use formula (\ref{eq:i2}) to calculate $J_6$ and the explicit form of the degree 4 invariant of the plane cubics to calculate $J_{12}$. Notice that Theorem \ref{thm:cubics-invs} is not sufficient, since non-smooth cubics do not admit a Weierstrass form. Nevertheless these are simple calculations which show that the only codimension$>1$ orbits outside the nullcone are $D,D^*,E,E^*$ with
\[  \begin{array}{cccc}
          & J_6 & J_{12} & k \\
       D   & 1 & 1 & 1 \\
       D^* & -8 & 16 & 4 \\
       E   & 1 & 1 & 1 \\
       E^* & -8 & 16 & 4 \\
     \end{array}  \]

 Since $k$ is a bijection on the codimension 1 orbits it is enough to find values of $c$ and $g$ with the given $k$. Since $k(\nu_{c,g})=\frac{12g^2}{3g^2-c}$, it is immediate that $k(B)=1$ for $B:=\nu_{-9,1}$ and $k(B^*)=4$ for $B^*:=\nu_{0,1}$ ($B$ and $B^*$ are notations from \cite{wall:nets}). Consequently the complete list of semistable orbits are $B,B^*,D,D^*,E,E^*$.

\subsubsection{The discriminant} For the representations $\noc$ and $S^3\C^3$ the nonstable variety is a hypersurface and we call their defining equation the {\em discriminant} of the representation. It is a classical result that the discriminant of the plane cubics is $\Delta=4a^3+27b^2$. Using the $c$-$g$-form we can quickly check that $-2^83^3\delta^*\Delta=(J_6^2-J_{12})^2(J_6^2-4J_{12})$ (see Section \ref{sec:gitmap} for the details), consequently a net is unstable if and only if its determinant cubic is unstable, but the $\delta^*$-image of the discriminant is not the discriminant: the component $(J_6^2-J_{12})=\overline B$ is counted with multiplicity~2.

\section{Hierarchy of the nets of conics} \label{sec:hierarchy}

C. T. C. Wall's result on the classification of the $\noc$-orbits can be verified using the results of Sections \ref{sec:class} and \ref{sec:inv-of-noc}. The codimension 1 orbits are classified by their $k$-invariant. The fact stated in Theorem \ref{thm:nets_equiv_classes}, namely that the restriction equations determine the equivariant classes imply that no orbit is missed in Table \ref{symmetries}. Any missing orbit would cause an indeterminacy in the solution of the restriction equations, hence would contradict to Theorem \ref{thm:nets_equiv_classes}.

To determine the hierarchy we use that a cohomologically defined incidence class determines adjacency for {\em positive} orbits: Consider a Lie group $G$ acting on a vector space $V$ complex linearly. For $v\in V$ let $T_v$ denote the maximal torus of the stabilizer subgroup of $G$.

\begin{definition}The orbit $Gv$ is positive if there is a linear functional $\varphi$ on the weight space of $T_v$ such that for all weights $w_i$ of the $T_v$-action on the normal space of the orbit $Gv$ at $v$ we have $\varphi(w_i)>0$.
\end{definition}

\begin{theorem} \cite{incidence} \label{thm:incidence} Let $\eta\subset V$ be a $G$-invariant subvariety and suppose that the orbit $Gv$ is positive for some $v\in V$. Then $v\in \eta$ if and only if $[\eta\subset V]_{T_v}\neq 0$.\end{theorem}

Table \ref{positivity} contains the description of normal slices to orbits, and their weights. The last column contains the values of the functional $\varphi$ (that is, the values $\varphi(\alpha), \varphi(\beta),\ldots$) if such a functional---proving the positivity of the given orbit---exists. By inspection we obtain the following fact.

\begin{proposition} All unstable orbits of $\noc$ are positive.\end{proposition}

\begin{table}\caption{Normal weights and positivity}\label{positivity}\begin{tabular}{|l|l|l|p{100mm}|l|}

\hline
$\Sigma^0$&&&normal weights&$\scriptstyle{\varphi(\alpha),\varphi(\beta),\dots}$\\
\hline
$C$         & 2 & $y^2+2xz,2yz,-x^2$           & $4\alpha-4\beta,2\alpha-2\beta$      & 1,0  \\
$D$         & 2 & $x^2, y^2, z^2+2xy$          & $3\alpha-3\beta,3\beta-3\alpha$      & -   \\
$D^*$       & 2 & $2xz, 2yz, z^2+2xy$          & $3\alpha-3\beta,3\beta-3\alpha$      & -   \\
$E$         & 3 & $x^2,y^2,z^2$                & $2\alpha-\beta-\gamma,2\beta-\alpha-\gamma,2\gamma-\alpha-\beta$              & - \\
$E^*$       & 3 & $2xy, 2yz, 2zx$              & $-2\alpha+\beta+\gamma,-2\beta+\alpha+\gamma,-2\gamma+\alpha+\beta$           &-\\
$F$         & 3 & $x^2+y^2, 2xy, 2yz$          & $2\alpha-2\beta,2\alpha-2\beta,\alpha-\beta$                                  & 1,0 \\
$F^*$       & 3 & $x^2+y^2,xz,z^2$             & $2\beta-2\alpha,2\beta-2\alpha,\beta-\alpha,$                                & -1,0 \\
$G$         & 4 & $x^2,y^2,yz$                 & $2\alpha-2\gamma,\beta-\alpha,2\beta-\alpha-\gamma,2\beta-2\gamma$            & 1,2,0 \\
$G^*$       & 4 & $xy, xz,z^2$                 & $2\gamma-2\alpha,\alpha-\beta,-2\beta+\alpha+\gamma,2\gamma-2\beta$           & -1,-2,0 \\
$H$         & 5 & $x^2,2xy,y^2+2xz$            & $2\alpha-2\beta,2\alpha-2\beta,3\alpha-3\beta,3\alpha-3\beta,4\alpha-4\beta$  & 1,0 \\
$I$         & 7 & $x^2,xy,y^2$                 & $\alpha-\gamma,2\alpha-\beta-\gamma,2\alpha-2\gamma,\beta-\gamma,-2\gamma+\alpha+\beta,
                                                     2\beta-2\gamma$                                                           & 0,0,-1 \\
$I^*$       & 7 & $xz,yz,z^2$                  & $-\alpha+\gamma,-2\alpha+\beta+\gamma,-2\alpha+2\gamma,-\beta+\gamma,2\gamma-\alpha-\beta,
                                                     -2\beta+2\gamma$                                                          & 0,0,1 \\
\hline
$\Sigma^1$&&&&\\
\hline
$(1^4)$  & 4 & $x^2-xz,y^2-yz,0$     & $\beta-2\gamma,\beta-2\gamma,\beta-2\gamma,\beta-2\gamma$                               & 1,0 \\
$(21^2)$  & 5 & $xy,xz+yz,0$          & $2\alpha-2\beta,\gamma-2\alpha,\gamma-2\alpha,\gamma-\alpha-\beta,\gamma-2\beta$        & 1,0,3 \\
$(31)$   & 6 & $xz,x^2-yz,0$         & $3\beta-3\alpha,2\beta-2\alpha,\gamma-3\alpha-\beta,\gamma-4\alpha,\gamma-2\alpha-2\beta,
                                         \gamma-4\beta$                                                                        & 0,1,5 \\
$(22)$   & 6 & $x^2,yz,0$            & $2\alpha-2\beta,2\alpha-2\gamma,\delta-\alpha-\beta,\delta-\alpha-\gamma,\delta-2\beta,\delta-2\gamma$
                                                                                                                               & 1,0,0,2 \\
$(4)$    & 7 & $xz+y^2,x^2,0$        & $2\alpha-2\beta,3\alpha-3\beta,4\alpha-4\beta,\gamma-3\alpha-\beta,\gamma-2\alpha-2\beta,
                                          \gamma-4\beta$                                                                       & 1,0,5 \\
$K$      & 8 & $y^2,z^2,0$           & $2\alpha-2\beta,2\alpha-\beta-\gamma,2\gamma-2\beta,2\gamma-\alpha-\beta,\delta-2\beta,$ &       \\
                                  &&&     $\delta-\alpha-\beta,\delta-\beta-\gamma,\delta-\alpha-\gamma$                       & 0,-1,0,0 \\
$L$      & 8 & $xy,xz,0$             & $\alpha-\beta,\alpha-\gamma,\alpha+\beta-2\gamma,\alpha+\gamma-2\beta,\delta-2\alpha,\delta-2\beta,$&\\
                                  &&&     $\delta-2\gamma,\delta-\beta-\gamma$                                   &1,0,0,3 \\
$M$      & 9 & $yz,y^2,0$           &$\beta+\gamma-2\alpha,\beta-\alpha,2\beta-2\alpha,2\beta-\alpha-\gamma,2\beta-2\gamma,\delta-2\alpha,$&\\
                                  &&&     $\delta-\alpha-\beta,\delta-\alpha-\gamma,\delta-2\gamma$                            & 0,1,0,2 \\
\hline
$\Sigma^2$&&&&\\
\hline
$S$      & 10&  $xy-z^2, 0,0$        & $\gamma-2\alpha-2\beta,\gamma-3\alpha-\beta,\gamma-4\alpha,\gamma-\alpha-3\beta,\gamma-4\beta,$&\\
                                  &&&  $\delta-2\alpha-2\beta,\delta-3\alpha-\beta,\delta-4\alpha,\delta-\alpha-3\beta,\delta-4\beta  $
                                                                                                                              & 0,0,1,1 \\
$PL$     & 11 & $xy,0,0$             & $\delta-2\alpha, \delta-2\beta,\delta-\alpha-\gamma,\delta-\beta-\gamma,\delta-2\gamma$     &  \\
         &    &                      & $\epsilon-2\alpha, \epsilon-2\beta,\epsilon-\alpha-\gamma,\epsilon-\beta-\gamma,\epsilon-2\gamma,\alpha+\beta-2\gamma$ & 1,0,0,3,3 \\
$DL$     & 13 & $x^2,0,0$            & $\delta-\alpha-\beta, \delta-2\beta,\delta-\alpha-\gamma,\delta-2\gamma,\delta-\beta-\gamma$  &  \\
         &    &                      & $\epsilon-\alpha-\beta, \epsilon-2\beta,\epsilon-\alpha-\gamma,\epsilon-2\gamma,\epsilon-\beta-\gamma$ & \\
         &    &                      & $2\alpha-2\beta, 2\alpha-\beta-\gamma, 2\alpha-2\gamma$ & 1,0,0,3,3\\
\hline

\end{tabular}\end{table}

Thus, Theorem \ref{thm:incidence} determines almost all adjacencies of the orbits, namely the ones involving unstable orbits. The missing adjacencies of the semistable orbits can be determined by calculating the $k$-invariant. As a result we obtain the complete hierarchy, depicted on Figure \ref{fig:hierarchy}.

\begin{example} \rm
Consider the orbits $F$ and $F^*$, and their adjacency with the orbit $(1^4)$. Let $v$ be the point in the $(1^4)$ orbit given in the Table, and let $j_{(1^4)}$ be the restriction homomorphism $H^*_{\GL(U)\times \GL(V)}(\noc)\to H^*_{T_v}(\noc)$. One can read from the table above that the homomorphism $j_{(1^4)}$ is the substitution
$$ c_i=\sigma_i(\alpha,\alpha,\alpha), \qquad d_i=\sigma_i(2\alpha,2\alpha,\beta),$$
where $\sigma_i$ denotes the $i$th elementary symmetric polynomial. For the equivariant classes given in Theorem \ref{thm:nets_equiv_classes} we have
$$j_{(1^4)} ([\overline{F}]) = 0, \qquad j_{(1^4)}([\overline{F^*}])=-6(2\alpha-\beta)(4\alpha^2-4\alpha\beta+\beta^2)\not=0.$$
Hence, we have that $(1^4)$ is contained in the orbit closure of $F^*$, but is not contained in the orbit closure of $F$.
\end{example}

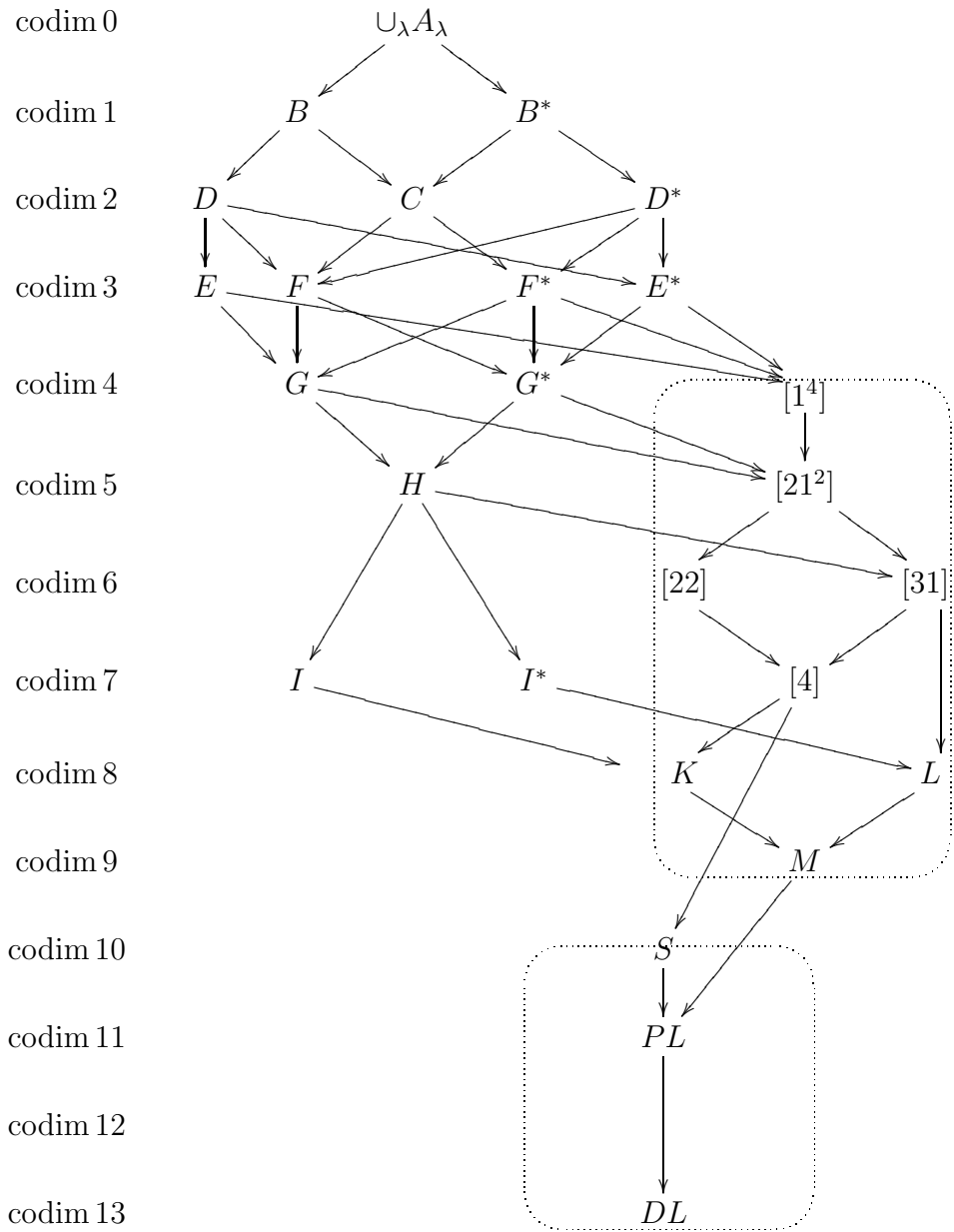
\begin{figure}
\xymatrix @=1.5pc{
\codim 0  &  & & \cup_\lambda A_{\lambda} \ar[dl] \ar[dr] & &         \\
\codim 1 &  & B \ar[dr]\ar[dl] &  & B^* \ar[dl]\ar[dr]        \\
\codim 2 & D\ar[d]\ar[dr]\ar[drrrr] & & C\ar[dl]\ar[dr] & & D^*\ar[dlll]\ar[d]\ar[dl]                    \\
\codim 3 & E \ar[dr] \ar[drrrrr] &  F \ar[drr]\ar[d]& &F^* \ar[dll]\ar[d]\ar[drr] & E^* \ar[dl]\ar[dr]                  \\
\codim 4 & & G\ar[dr]\ar[drrrr] & & G^*\ar[dl]\ar[drr] &
         & *\txt{\phantom{x}\\ $[1^4]$} \save [l].[dddddl].[dddddr].[r]*[F.:<15pt>]\frm{}="back" \restore \ar[d]     &             \\
\codim 5 & & & H \ar[ddr]\ar[ddl] \ar[drrrr] & & & [21^2]\ar[dr]\ar[dl]        &           \\
\codim 6 & & & & & \ \ \ \ [22]\ar[dr] & &  [31]\ \ \ \ar[dl] \ar[dd]          &          \\
\codim 7 & & I \ar[drrr]& & I^*\ar[drrr] & & [4] \ar[dl]     \ar[dddl]           &   \\
\codim 8 & & & & & \ \ \ \ K\ar[dr]  & &  L\ \ \ar[dl]           &          \\
\codim 9 & & & & &   & M \ar[ddl] &            &          \\
\codim 10 & & &  & & S \ar[d] \save [l].[dddl].[dddr].[r]*[F.:<15pt>]\frm{}="back" \restore &  &            &          \\
\codim 11 & & & & & PL \ar[dd] &  &            &          \\
\codim 12 & & & & & & & & \\
\codim 13 & & & & & DL  &  &            &
}
\caption{The hierarchy of $\noc$ orbits}\label{fig:hierarchy}
\end{figure}

\subsection{The equivariant cohomology rings of the orbits}\label{sudoku}
With a little extra sudoku type calculations one can determine the equivariant cohomology rings $H^*(B{G_x})$ of the orbits $Gx$. In column 5 of Table \ref{symmetries} we listed the degrees of a free generating set for these rings. This information can be used e.g. to define certain ``higher'' Thom polynomials. Since the equivariant cohomology spectral sequence of the codimension filtration (Kazarian spectral sequence, see \cite{kazass}, \cite[Sect.10]{cr}) degenerates, the Poincar\'e series of the rings $H^*(BG_x)$ shifted by the codimension add up to the Poincar\'e series of $H^*(BG)$. For the open stratum $O=\bigcup_{\mu\in \P^1} A_\mu$ we have $H^*_{G}(O)=H^*_{\gl_1}(\P^1)$ for the trivial $\gl_1$-action, so the  Poincar\'e series of $H^*_{G}(O)$ is $\frac{1+t}{1-t}$, and we get that
\[\frac{1+t}{1-t}+\frac{t^2}{(1-t)^2}+\frac{2t^2}{(1-t)(1-t^2)}+\ldots+
\frac{t^{18}}{(1-t)^2(1-t^2)^2(1-t^3)^2}=\frac1{(1-t)^2(1-t^2)^2(1-t^3)^2}.\]
Here the $t$-exponents of the numerator are the codimensions (column 2) and the $t$-exponents of the denominators are degrees of the generators of $H^*(BG_x)$ (column 5).

\section{Enumerative questions, multiplicities of the determinant map} \label{sec:mult}
Equivariant classes contain enumerative data in a compressed form. One of the simplest enumerative invariants is the degree. Using \cite[Section 6]{forms} we see that the degree of a $G$-invariant subvariety  $\eta\subset\noc$ can be obtained by substituting 1 to the Chern roots of $\GL(V)$ and 0 to the Chern roots of $\GL(U)$ in the polynomial $[\eta]\in H^*(B(\GL(U)\times \GL(V)))$. For example, we obtain
\[ \deg(\overline{C})=72,\qquad \deg(\overline{D})=36,\qquad \deg(\overline{D^*})=45.  \]
Consequently in a generic two-parameter linear family of nets of conics there are 72 of type $C$, 36 of type $D$ and 45 of type $D^*$, etc.

A more subtle enumerative invariant is the {\em intersection multiplicity}. Suppose that $f:X\to Y$ is a map of smooth varieties and $Z\subset Y$ a codimension $d$ subvariety. Suppose also that the preimage $f^{-1}Z$ is pure $d$-codimensional. Then we can assign positive integers $\mu_i$ to every component $Y_i$ of $f^{-1}Z$ called intersection multiplicities (for more details see \cite{fulton:intersection}). If $X=Y=\C$ and $Z=\{0\}$, then the intersection multiplicities are the usual multiplicities of the roots of $f$.

An important property of the intersection multiplicity (\cite[Sect. 7]{fulton:intersection}) is that if all components $Y_i$ of $f^{-1}Z$ are of codimension $d$, then
\begin{equation}\label{eq:multiplicity} f^*([X])=\sum \mu_i[Y_i].  \end{equation}
In general this equation does not determine the intersection multiplicities $\mu_i$. However (\ref{eq:multiplicity}) generalizes to the equivariant setting where there is more chance that the classes $[Y_i]$ will be linearly independent.

Consider the determinant map $\delta:\noc\to S^3 V$ studied in Section \ref{sec:determinant}. From the normal forms in Table \ref{symmetries} it is easy to calculate the image under the determinant map. E.g. for $C$ the normal form is $y^2+2xz,2yz,-x^2$ which means that the generic net in matrix form is
\[  \left(
  \begin{array}{ccc}
    -\kappa & \cdot & \lambda \\
    \cdot & \lambda & \mu \\
    \lambda & \mu & \cdot \\
  \end{array}
\right), \]
with determinant $\mu^2\kappa-\lambda^3$ corresponding to the cuspidal curve $\nu$. A table can be found in \cite{wall:nets}. For the readers' convenience we included the images of the determinant map in the last column of Table \ref{symmetries}. The notation tries to imitate the shape of the various classes of plane cubics, i. e.  $\theta$ is the orbit (closure) of conic intersected by line, $\Omega$ the conic and tangent line, $\mathsf{A}$ the three nonconcurrent lines, $\neq$ is the double line intersected by a third line, $\Xi$ is the triple line and $\Zhe$ is the three concurrent lines. The codimension 1 orbits will be treated in the next section.

The equivariant classes of the $\GL_3$-representation $S^3\C^3$ (ie. plane cubics) were calculated by B.~K\H om\H uves \cite{komuves:thesis}. The cases we need are
\begin{equation}\label{eq:balazs_classes}
 [\nu]=24e_1^2,\ [\theta]=18e_1^2+9e_2,\ [\Omega]=36e_1^3+18e_1e_2,\ [\mathsf{A}]=12e_1^3+6e_1e_2+27e_3,\  [\Zhe]=e_1[\mathsf{A}],
\end{equation}
where $e_i$ denote the $\gl_3$-Chern classes.

The map $\delta$ is $\GL_3\times \GL_3$-equivariant, if we replace the $\gl_3$-action on the plane cubics by the $\gl(U)\times\gl(V)$-action as in Section \ref{sec:determinant}. The effect of this change on (\ref{eq:balazs_classes}) is replacing the $\gl_3$-Chern roots $\epsilon_i$ ($\epsilon_1+\epsilon_2+\epsilon_3=e_1$ etc) by some linear combination of the $\gl(U)\times\gl(V)$-Chern roots. A comparison of the actions of the maximal tori of $\GL_3$ and $\GL(U)\times GL(V)$ gives the substitution $\epsilon_i\mapsto \delta_i-2/3u_1$ (where $\delta_i$ denote the $\gl(V)$-Chern roots). Consequently for the Chern classes we obtain the substitution $e_1\mapsto v_1-2u_1$, $e_2\mapsto v_2-4/3u_1v_1+4/3u_1^2$ and $e_3\mapsto v_3-2/3v_2u_1+4/9v_1u_1^2-8/27u_1^3$.

After this substitution we have all the ingredients of (\ref{eq:multiplicity}) for the $\GL(U)\times \GL(V)$-equivariant map $\delta$. Then explicit calculation implies the following theorem.

\begin{theorem}\label{multiplicities} In $\GL(U)\times \GL(V)$-equivariant cohomology we have
   \[   \delta^*([\nu])=3[C],  \]
   \[   \delta^*([\theta])=4[D]+[D^*],  \]
   \[   \delta^*([\Omega])=9[F]+6[F^*],  \]
   \[   \delta^*([\mathsf{A}])=8[E]+[E^*]+2[F],  \]
and the coefficients on the right hand side are uniquely determined, therefore they are the intersection multiplicities.
\end{theorem}

For the three concurrent lines the codimension condition is {\em not} satisfied, but a similar calculation gives the equality
   \[   \delta^*([\Zhe])=12[(1^4)]+\big(4[G]+1/2[G^*]+2(d_1-2c_1)[F]\big),  \]
and, like above, the coefficients are uniquely determined. The class in the bracket is supported on $\overline{F}$, so we can call 12 the intersection multiplicity of $(1^4)$.

\begin{remark}\label{alg-mult}
The intersection multiplicities are always at most the algebraic multiplicities by \cite[Pr. 8.2]{fulton:intersection} and they agree if the image is (locally) Cohen-Macaulay and the preimage has the same codimension by \cite[p.108]{fulton-pragacz}. In \cite{chipalkatti} J. Chipalkatti  determines which  orbit closures are arithmetically Cohen-Macaulay for the plane cubics. Among the orbits of Theorem \ref{multiplicities} $\nu$ is Cohen-Macaulay, since it is a complete intersection, $\Omega$  is arithmetically Cohen-Macaulay consequently Cohen-Macaulay, but $\theta$ and $\mathsf{A}$ are not arithmetically Cohen-Macaulay. Therefore the intersection multiplicities for $\nu$ and $\Omega$ are algebraic multiplicities as well. For $\theta$ and $\mathsf{A}$ we do not know if they are Cohen-Macaulay. If  the algebraic multiplicities for $\theta$ and $\mathsf{A}$ differ from the intersection multiplicities, then they cannot be Cohen-Macaulay. Unfortunately we were not able to calculate these algebraic multiplicities.
\end{remark}

For the codimension 1 orbits we study the induced map on the GIT quotients.

\subsection{The induced map on the GIT quotients  of  nets of conics to plane cubics} \label{sec:gitmap}
To see that the $G=\GL(U)\times \GL(V)$-equivariant determinant map $\delta$ induces a map  $d$ of the corresponding GIT quotients we need to check that semistable orbits are mapped to semistable orbits. It follows from general principles but here is a direct verification. For codimension $>1$ orbits of $\noc$ we see this fact from Table \ref{symmetries}. (For the plane cubics the semistable orbits are the smooth curves which are also stable, together with the nodal curve and $\theta$ and $\mathsf{A}$.) The codimension 1 orbits $A_\mu$ have a $\nu_{c,g}$ representative with $(c,g)\neq(0,0)$, so for $\delta(\nu_{c,g})$ either $a=c-3g^2$ or $b=2g(c+g^2)$ is not zero, therefore the image is semistable.

The traditional parametrization of the GIT quotient $S^3\C^3//G$ is the $j$-invariant $j=\frac{4a^3}{\Delta}: S^3\C^3//G\to \P^1$, where $\Delta=4a^3+27b^2$ is the discriminant. On the other hand we saw in Section~\ref{sec:stability} that the invariant $k=\frac{J_6^2}{J_{12}}:\noc//G\to \P^1$ parametrizes the GIT quotient $\noc//G$. Using these parametrizations we can calculate $\tilde d:=j\circ d\circ k^{-1}:  \P^1\to \P^1$:

Using the $c$-$g$-form we get (by some abuse of notation)
\[ \Delta=\Delta(\nu_{c,g})=4a^3+27b^2=4a^3+27(4g^2)(a+4g^2)^2=4(a+3g^2)(a+12g^2)^2, \]
and we have $J_6=24g$ and $J_{12}=-48a$, so $48^3\Delta=(-4J_{12}+J_6^2)(-J_{12}+J_6^2)^2$. Therefore $j\circ d=\frac{J_{12}^3}{(J_6^2-4J_{12})(J_6^2-J_{12})^2}$ and
   \[ \tilde{d}(x:y)=(4y^3:(x-4y)(x-y)^2). \]

 The branching points are at the singular points (take the affine chart $y=1$, and find the zeroes of $3(x-1)(x-3)$---the derivative of $(x-4)(x-1)^2$---and an extra branching at $x=\infty$.):
\begin{itemize}
\item[$x=1$]  ($\Delta=0$ i.e $j$-invariant is $\infty$) corresponds to the semistable point (class of the nodal curve). It has preimages $B$ with $k(B)=1$ and $B^*$ with $k(B^*)=4$. The point $B$ has multiplicity two. Notice that $B$ and $B^*$ represent the two semistable points in $\noc//G$ (in the GIT quotient $B\sim D\sim E$ and $B^*\sim D^*\sim E^*$).
\item[$x=\infty$]  ($j$-invariant is $0$) corresponds to the degree 4 orbit of elliptic curves defined by $a=0$. It has one preimage (of multiplicity 3): the orbit of nets of conics defined by $J_{12}=0$.
\item[$x=3$]  ($j$-invariant is $1$) corresponds to the degree 6 orbit of elliptic curves defined by $b=0$. It has one preimage with $k=3$ and another one with multiplicity 2, the degree 6 orbit of nets of conics defined by $J_6=0$.
\end{itemize}

\begin{remark}\label{hyper-mult}  For hypersurfaces the intersection multiplicities agree with the algebraic (or scheme theoretic) multiplicities. To determine these multiplicities in our case it is enough to compare degrees. One obtains that all algebraic multiplicities are 1, except for $B$ which has multiplicity~2.
\end{remark}

\bibliography{noc}
\bibliographystyle{alpha}

\end{document}